\documentclass[12pt,a4paper,reqno]{amsart}
\usepackage{amsmath, amssymb, amscd, amsthm, amsbsy, euscript} 
\usepackage{textcomp,txfonts}
\usepackage{graphicx,bbm}

\newtheorem{theor}{Theorem}
\newtheorem{claim}[theor]{Claim}
\newtheorem*{conjecture}{Conjecture}

\newtheorem{cor}[theor]{Corollary}

\theoremstyle{definition}

\newtheorem{example}{Example}

\theoremstyle{remark}
\newtheorem{rem}{Remark}

\newcommand{\cev}[1]{\reflectbox{\ensuremath{\vec{\reflectbox{\ensuremath{#1}}}}}}
\newcommand{\circL}{\ \cev{\circ}\ }

\newcommand{\BBR}{\mathbb{R}}

\newcommand{\BBQ}{\mathbb{Q}}

\newcommand{\cF}{\mathcal{F}}
\newcommand{\bcF}{\boldsymbol{\cF}}

\newcommand{\cP}{\mathcal{P}}\newcommand{\cQ}{\mathcal{Q}}

\newcommand{\dd}{\partial}
\newcommand{\Id}{{\mathrm d}}

\DeclareMathOperator{\Aut}{Aut}

\DeclareMathOperator{\Assoc}{Assoc}
\DeclareMathOperator{\Jac}{Jac}

\newcommand{\poly}{{ \text{\textrm{\textup{poly}}} }}

\newcommand{\schouten}[1]{\lshad {#1} \rshad}

\DeclareMathOperator{\Star}{Star}

\newcommand{\lshad}{[\![}
\newcommand{\rshad}{]\!]}

\newcommand{\by}[1]{\textit{{#1}}}
\newcommand{\jour}[1]{\textit{{#1}}}
\newcommand{\vol}[1]{\textbf{{#1}}}
\newcommand{\book}[1]{\textrm{{#1}}}

\voffset=-
6mm

\title[Formality morphism as the mechanism of $\star$-\/product associativity]{Formality morphism as the mechanism of\\[3pt] $\star$-\/product associativity: how it works}

\author[R.~Buring]{Ricardo Buring${}^{\text{1)}}$}
\thanks{${}^{\text{1)}}$\:\textit{Address}: 
Institut f\"ur Mathematik, 
Johannes Gutenberg\/--\/Uni\-ver\-si\-t\"at,
Staudingerweg~9, 
\mbox{D-\/55128} Mainz, Germany.
\quad 
\textit{E-mail} (corresponding author): \texttt{rburing\symbol{"40}uni-mainz.de}
}

\author[A.\,V.\,Kiselev]{Arthemy V. Kiselev${}^{\text{2)}}$}
\thanks{${}^{\text{2)}}$\:\textit{Address}: Bernoulli Institute for Mathematics,
Computer Science and Artificial Intelligence, 
University of Groningen,
P.O.Box~407, 9700\,AK Groningen, 
The Netherlands.\quad \textit{E-mail}:\ \texttt{A.V.Kiselev\symbol{"40}rug.nl}
}


\date{1 July 2019}

\dedicatory{
`Symmetries \textsl{\&}\ integrability of equations of mathematical physics',\\
\textup{(22--24}~December \textup{2018}, IM NASU Kiev, Ukraine\textup{)}%
}

\begin{document}
\begin{abstract}
The formality morphism~$\boldsymbol{\cF}=\{\cF_n$, $n\geqslant1\}$
in Kontsevich's deformation quantization
is a collection of maps from tensor powers of
the differential graded Lie algebra (dgLa) of multivector fields to the 
dgLa of polydifferential operators on finite\/-\/dimensional affine manifolds. Not a Lie algebra morphism by its term~$\cF_1$ alone, 
the entire set $\boldsymbol{\cF}$ 
is an $L_\infty$-\/morphism instead. 
It 
induces a map of the Maurer\/--\/Cartan elements, 
taking Poisson bi\/-\/vectors to deformations $\mu_A\mapsto\star_{A[[\hbar]]}$ of the usual multiplication 
of functions into associative noncommutative $\star$-\/products of power series in~$\hbar$.
The associativity of $\star$-\/products is then realized, in terms of the Kontsevich graphs which encode polydifferential operators, by differential consequences of the Jacobi identity.
The aim of this paper is to illustrate the work of this algebraic mechanism for the Kontsevich $\star$-\/products 
(in particular, with harmonic propagators). We inspect how the Kontsevich weights are correlated 
for the orgraphs which occur in the associator for~$\star$ and in its expansion using Leibniz graphs with the Jacobi identity at a vertex.
\end{abstract}
\maketitle

\subsection*{Introduction}
The Kontsevich formality morphism~$\bcF$ relates two differential graded Lie algebras (dgLa). Its domain of definition is the shifted\/-\/graded vector space $T_\poly^{\downarrow[1]}(M^r)$ of multivectors on an affine real finite\/-\/dimensional manifold~$M^r$; the graded Lie algebra structure is the Schouten bracket~$\lshad\,,\,\rshad$ and the differential is set to (the bracket with) zero by definition.
On the other hand, the target space of the formality morphism~$\bcF$ is the graded vector space $D_\poly^{\downarrow[1]}(M^r)$ of polydifferential operators on~$M^r$; the graded Lie algebra structure is the Gerstenhaber bracket~$[\,,\,]_G$ and the differential $\Id_H=[\mu_A,\cdot]$ is induced by using the multiplication~$\mu_A$ in the algebra $A\mathrel{{:}{=}} C^\infty(M^r)$ of functions on~$M^r$.
It is readily seen that w.r.t.\ the above notation, Poisson bi\/-\/vectors~$\cP$ satisfying the Jacobi identity $\lshad\cP,\cP\rshad=0$ on~$M^r$ are the Maurer\/--\/Cartan elements (indeed, $(d\equiv0)(\cP)+\tfrac{1}{2}\lshad\cP,\cP\rshad=0$).  Likewise, for a (non)commutative star\/-\/product $\star=\mu_{A[[\hbar]]}+\langle\text{tail}\mathrel{{=}{:}} B \rangle$, which deforms the usual multiplication $\mu=\mu_{A[[\hbar]]}$ in $A[[\hbar]]=C^\infty(M^r)\mathbin{{\otimes}_\BBR} \BBR[[\hbar]]$ by a tail~$B$ w.r.t.\ a formal parameter~$\hbar$, the requirement that $\star$~be associative again is the Maurer\/--\/Cartan equation,
\[
[\mu,B]_G+\tfrac{1}{2}[B,B]_G=0 \qquad \Longleftrightarrow \qquad
\tfrac{1}{2} [\mu+B,\mu+B]_G=0.
\]
Here, the leading order equality $[\mu,\mu]_G=0$ expresses the given associativity of the product~$\mu$ itself.

The Kontsevich formality mapping $\bcF=\{\cF_n\colon T_\poly^{\otimes n}\to D_\poly$, $n\geqslant 1\}$ in~\cite{Ascona96,KontsevichFormality} is an $L_\infty$-\/morphism which induces a map 
that takes Maurer\/--\/Cartan elements~$\cP$, i.e.\ formal Poisson bi\/-\/vectors~$\tilde{\cP}=\hbar\cP+\bar{o}(\hbar)$ on~$M^r$, to Maurer\/--\/Cartan elements\footnote{%
In fact, the morphism~$\bcF$ is a quasi\/-\/isomorphism 
(see~\cite[Th.~6.3]{KontsevichFormality}),
inducing a bijection between the sets of gauge\/-\/equivalence classes of Maurer\/--\/Cartan elements.},
i.e.\ the tails~$B$ in solutions~$\star$ of the associativity equation on~$A[[\hbar]]$.

The theory required to build the Kontsevich map~$\bcF$ is standard, well reflected in the literature (see~\cite{Ascona96,KontsevichFormality}, as well as~\cite{CattaneoIndelicato,CattaneoKellerTorossian} and references therein); a proper choice of signs is analysed in~\cite{ArnalManchonMasmoudi,WillwacherCalaque}. 
The framework of homotopy Lie algebras and $L_\infty$-\/morphisms, introduced by Schlessinger\/--\/Stasheff~\cite{SchlessingerStasheff}, is available from~\cite{LadaStasheff}, cf.~\cite{CattaneoFelder2000} in the context of present paper.

So, the general fact of (existence of) factorization,
\begin{equation}\label{EqDiamondAssoc}
\Assoc(\star)(\cP)(f,g,h) = \Diamond\bigl(\cP,\lshad\cP,\cP\rshad\bigr)(f,g,h),\qquad f,g,h\in A[[\hbar]],
\end{equation}
is 
known to the expert community. Indeed, this factorization is immediate from the construction of $L_\infty$-\/morphism 
in~\cite[\S6.4]{KontsevichFormality}.
We shall inspect how this mechanism works in practice, i.e.\ how precisely the $\star$-\/product is made associative in its perturbative expansion whenever the bi\/-\/vector~$\cP$ is Poisson, thus satisfying the Jacobi identity $\Jac(\cP)\mathrel{{:}{=}}\tfrac{1}{2}\lshad\cP,\cP\rshad=0$. To the same extent as our paper~\cite{OrMorphism} justifies a similar factorization, $\lshad\cP,\cQ(\cP)\rshad=\Diamond\bigl(\cP,\lshad\cP,\cP\rshad\bigr)$, of the Poisson cocycle condition for universal deformations $\dot{\cP}=\cQ(\cP)$ of Poisson structures\footnote{Universal w.r.t.\ all Poisson brackets on all finite\/-\/dimensional affine manifolds, such infinitesimal deformations were pioneered in~\cite{Ascona96}; explicit examples of these flows $\dot{\cP}=\cQ(\cP)$ are given in~\cite{f16,sqs17,OrMorphism}.},
we presently motivate 
the findings in~\cite{sqs15} for $\star$ mod~$\bar{o}(\hbar^3)$, proceeding to the next order $\star$ mod~$\bar{o}(\hbar^4)$ from~\cite{cpp} (and higher orders, recently available from~\cite{BanksPanzerPym1812}).\footnote{Note that both the approaches --\,to noncommutative associative $\star$-\/products and deformations of Poisson structures\,-- rely on the same calculus of oriented graphs by Kontsevich~\cite{MKParisECM,
Ascona96,KontsevichFormality}.%
}
Let us emphasize that the theoretical constructions and algorithms (contained in the computer\/-\/assisted proof scheme under study and in the tools for graph weight calculation) would still work at arbitrarily high orders of expansion~$\star$ mod~$\bar{o}(\hbar^k)$ as~$k\to\infty$. 
Explicit factorization~\eqref{EqDiamondAssoc} up to~$\bar{o}(\hbar^k)$ helps us build the star\/-\/product~$\star$ mod~$\bar{o}(\hbar^k)$ by using a self\/-\/starting iterative process, 
because 
the Jacobi identity for~$\cP$ is the only obstruction to the associativity of~$\star$. 
Specifically, the Kontsevich weights of graphs on fewer vertices (yet with a number of edges 
such that they do not show up in the perturbative expansion of~$\star$)
dictate the coefficients of Leibniz orgraphs in operator~$\Diamond$ at higher orders in~$\hbar$. These weights in the r.-h.s.\ of~\eqref{EqDiamondAssoc} constrain the higher\/-\/order weights of the Kontsevich orgraphs in the expansion of $\star$-\/product 
itself.
This is important also in the context of a number\/-\/theoretic open problem about the (ir)rational value $(\text{const}\in\BBQ\setminus\{0\} )\cdot\zeta(3)^2/\pi^6+(\text{const}\in\BBQ)$ of a graph weight at~$\hbar^7$ in~$\star$ (see~\cite{FelderWillwacher2008} and~\cite{BanksPanzerPym1812}).

Our paper is structured as follows. First, we fix notation and recall some basic facts from relevant theory. Secondly, we provide three examples which illustrate the work of formality morphism in solving Eq.~\eqref{EqDiamondAssoc}. Specifically, we read the operators $\Diamond_k=\Diamond$ mod~$\bar{o}(\hbar^k)$ satisfying
\begin{equation}\tag{\ref{EqDiamondAssoc}${}'$}
\Assoc(\star)(\cP)(f,g,h)\ \text{mod}\:\bar{o}(\hbar^k)
= \Diamond_k\,\bigl(\cP,\lshad\cP,\cP\rshad\bigr)(f,g,h)
\end{equation}
at $k=2$, $3$, and~$4$. This corresponds to the expansions~$\star$ mod~$\bar{o}(\hbar^k)$ in~\cite{KontsevichFormality}, \cite{sqs15}, and~\cite{cpp}, respectively. One can then continue with $k=5,6$; these expansions are in~\cite{BanksPanzerPym1812}. Independently, one can probe such factorizations using other stable formality morphisms: for instance, the ones which correspond to a different star\/-\/product, the weights in which are determined by a logarithmic propagator instead of the harmonic 
one (see~\cite{RossiWillwacher2014}).

\section{Two differential graded Lie algebra structures}
\noindent%
Let $M^r$ be an $r$-\/dimensional affine real manifold (we set $\Bbbk = \mathbb{R}$ for simplicity).
In the algebra $A \mathrel{{:}{=}} C^\infty(M^r)$ of smooth functions, denote by~$\mu_A$ (or equivalently, by the dot~$\cdot$) the usual commutative, associative, bi\/-\/linear multiplication.
The space of formal power series in~$\hbar$ over~$A$ will be~$A[[\hbar]]$ and the $\hbar$-\/linear multiplication in it is~$\mu$ (instead of~$\mu_{A[[\hbar]]}$).
Consider two differential graded Lie algebra stuctures.
First, we have that the shifted\/-\/graded space $T_{\poly}^{\downarrow[1]}(M^r)$ of multivector fields on~$M^r$ is equipped with the shifted\/-\/graded skew\/-\/symmetric Schouten bracket $\lshad\,,\,\rshad$ (itself bi\/-\/linear by construction and satisfying the shifted\/-\/graded Jacobi identity); the differential is set to zero.
Secondly, the vector space $D_{\poly}^{\downarrow[1]}(M^r)$ of polydifferential operators (linear in each argument but not necessarily skew over the set of arguments or a derivation in any of them) is graded by using the number of arguments~$m$: by definition, let $\deg(\theta(m\text{ arguments})) \mathrel{{:}{=}} m-1$.
For instance, $\deg(\mu_A) = 1$.
The Lie algebra structure on $D_{\poly}^{\downarrow[1]}(M^r)$ is the Gerstenhaber bracket $[\,,\,]_G$; for two homogeneous operators~$\Phi_1$ and~$\Phi_2$ it equals $[\Phi_1,\Phi_2]_G = \Phi_1 \circL \Phi_2 - (-)^{\deg \Phi_1 \cdot \deg \Phi_2}\Phi_2 \circL \Phi_1$, where the directed, non\/-\/associative insertion product is, by definition \[ (\Phi_1 \circL \Phi_2)(a_0,\ldots,a_{k_1+k_2}) = \sum_{i=0}^{k_1} (-)^{ik_2} \Phi_1\bigl(a_0 \otimes \ldots \otimes a_{i-1} \otimes \Phi_2(a_i \otimes \ldots \otimes a_{i+k_2}) \otimes a_{i+k_2+1} \otimes \ldots \otimes a_{k_1+k_2}\bigr). \]
In the above, $\Phi_i\colon A^{\otimes(k_i+1)}\to A$ so that $a_j\in A$.
Like $\schouten{\cdot,\cdot}$, the Gerstenhaber bracket satisfies the shifted\/-\/graded Jacobi identity.
The Hochshild differential on~$D_{\poly}^{\downarrow[1]}(M^r)$ is $\Id_H = [\mu_A, \cdot]_G$; indeed, its square vanishes, $\Id_H^2 = 0$, due to the Jacobi identity for~$[\,,\,]_G$ into which one plugs the equality $[\mu_A, \mu_A]_G = 0$.

\begin{example}
The associativity of the product~$\mu_A$ in the algebra of functions~$A = C^\infty(M^r)$ is the statement that
\begin{multline*}
\mu_A^{(1)}(\mu_A^{(2)}(a_0,a_1),a_2) + (-1)^{(i=1)\cdot (\deg \mu_A = 1)} \mu_A^{(1)}(a_0,\mu_A^{(2)}(a_1,a_2)) 
\\
 -(-)^{ (\deg \mu_A^{(1)} = 1) \cdot (\deg \mu_A^{(2)} = 1) } \bigl\{ \mu_A^{(1)}(\mu_A^{(1)}(a_0,a_1),a_2) - \mu_A^{(2)}(a_0,\mu_A^{(1)}(a_1,a_2)) \bigr\} 
\\
= 2\bigl\{(a_0\cdot a_1)\cdot a_2 - a_0\cdot(a_1\cdot a_2)\bigr\} = 0.
\end{multline*}
So, the associator $\Assoc(\mu_A)(a_0,a_1,a_2) = \frac{1}{2}[\mu_A,\mu_A]_G\,(a_0,a_1,a_2) = 0$ for any $a_j\in A$.
\end{example}

\section{The Maurer\/--\/Cartan elements}
\noindent%
In every 
differential graded Lie algebra with a Lie bracket~$[\,,\,]$, the Maurer\/--\/Cartan (MC) elements are solutions of degree~$1$ for the Maurer\/--\/Cartan equation
\begin{equation}
\label{EqMC}
\Id\alpha + \tfrac{1}{2}[\alpha, \alpha]=0,
\end{equation}
where $\Id$ is the differential (equal, we recall, to zero identically on $T_{\poly}^{\downarrow[1]}(M^r)$ and $\Id_H = [\mu_A, \cdot]_G$ on $D_{\poly}^{\downarrow[1]}(M^r)$. Likewise, the Lie algebra structure$[\cdot,\cdot]$ is the Schouten bracket $\schouten{\cdot,\cdot}$ and Gerstenhaber bracket $[\cdot,\cdot]_G$, respectively.)

Now tensor the degree\/-\/one parts of both 
dgLa structures with $\hbar \cdot \Bbbk[[\hbar]]$, i.e.\ with formal power series starting at~$\hbar^1$, and, preserving the notation (that is, extending the brackets and the differentials by $\hbar$-\/linearity), consider the same Maurer\/--\/Cartan equation~\eqref{EqMC}.
Let us study its formal power series solutions $\alpha = \hbar^1 \alpha_1 + \cdots$. 

So far, in the Poisson world we have that the Maurer\/--\/Cartan bi\/-\/vectors are formal Poisson structures $0+\hbar\cP_1 + \bar{o}(\hbar)$ satisfying~\eqref{EqMC}, which is $\schouten{\hbar\cP_1 + \bar{o}(\hbar), \hbar\cP_1 + \bar{o}(\hbar)} = 0$ with zero differential.
In the world of associative structures, the Maurer\/--\/Cartan elements are the tails~$B$ in expansions $\star = \mu + B$, so that the associativity equation $[\star, \star]_G=0$ reads (for $[\mu,\mu]_G = 0$) 
\[
[\mu,B]_G + \tfrac{1}{2}[B,B]_G = 0, \] which is again~\eqref{EqMC}.

\section{The $L_\infty$-\/morphisms} 
\noindent%
Our 
goal is to have (and use) a morphism $T_{\poly}^{\downarrow[1]}(M^r) \to D_{\poly}^{\downarrow[1]}(M^r)$ which would induce a map that takes Maurer\/--\/Cartan elements in the Poisson world to Maurer\/--\/Cartan elements in the associative world.

The leading term~$\mathcal{F}_1$, i.e. the first approximation to the morphism which we consider, is the Hochschild\/--\/Kostant\/--\/Rosenberg (HKR) map (obviously, extended by linearity), 
\[ \mathcal{F}\colon \xi_1 \wedge \ldots \wedge \xi_m \mapsto \frac{1}{m!} \sum\nolimits_{\sigma \in S_m} (-)^\sigma \xi_{\sigma(1)} \otimes \ldots \otimes \xi_{\sigma(m)}, \] which takes a split multi\/-\/vector to a polydifferential operator (in fact, an $m$-\/vector).
More explicitly, we have that
\begin{equation}\label{EqHKR}
\mathcal{F}_1\colon (\xi_1 \wedge \ldots \wedge \xi_m) \mapsto \bigg( a_1 \otimes \ldots \otimes a_m \mapsto \frac{1}{m!} \sum\nolimits_{\sigma \in S_m} (-)^\sigma \prod\nolimits_{i=1}^m \xi_{\sigma(i)}(a_i) \bigg),
\end{equation}
here $a_j \in A \mathrel{{:}{=}} C^\infty(M^r)$.
For zero\/-\/vectors $h\in A$, one has $\cF_1\colon h\mapsto(1\mapsto h)$.

\begin{claim}[{\cite[\S4.6.2]{KontsevichFormality}}]
The leading term, map~$\mathcal{F}_1$, is \emph{not} a Lie algebra morphism (which, if it were, would take the Schouten bracket of multivectors to the Gerstenhaber bracket of polydifferential operators).
\end{claim}

\begin{proof}[Proof (by counterexample)]
Take two bi\/-\/vectors;
their Schouten bracket is a tri\/-\/vector, but the Gerstenhaber bracket of two bi\/-\/vectors is a differential operator which has homogeneous components of differential orders $(2$,$1$,$1)$ and $(1$,$1$,$2)$.
And in general, those components do not vanish.
\end{proof}

The construction of not a single map~$\mathcal{F}_1$ but of an entire collection $\bcF = \{\mathcal{F}_n$, $n\geqslant 1\}$ of maps 
does nevertheless yield a well\/-\/defined mapping of the Maurer\/--\/Cartan elements from the two differential graded Lie algebras.\footnote{%
The name `Formality' for the collection~$\bcF$ of maps
is motivated by Theorem~4.10 in~\cite{KontsevichFormality} and by the main theorem in \textit{loc.~cit.}%
}

\begin{theor}[{\cite[Main Theorem]{KontsevichFormality}}]
\label{ThMainMK97}
There exists a collection of linear maps $\bcF = \{\cF_n \colon T_{\poly}^{\downarrow[1]}(M^r)^{\otimes n} \to D_{\poly}^{\downarrow[1]}(M^r)$, $n \geqslant 1\}$ such that $\mathcal{F}_1$~is the HKR map~\eqref{EqHKR} and $\bcF$~is an $L_\infty$-\/morphism of the two differential graded Lie algebras\textup{:} $\bigl(T_{\poly}^{\downarrow[1]}(M^r)$, $\schouten{\cdot,\cdot}$, $d=0\bigr) \to \bigl(D_{\poly}^{\downarrow[1]}(M^r)$, $[\cdot,\cdot]_G$, $\Id_H=[\mu_A,\cdot]_G\bigr)$.
Namely,
\begin{enumerate}
\item each component $\mathcal{F}_n$ is homogeneous of own grading~$1-n$,
\item\label{LinftySkew} 
each morphism $\mathcal{F}_n$ is graded skew\/-\/symmetric, i.e. \[ \mathcal{F}_n(\ldots,\xi,\eta,\ldots) = -(-)^{\deg(\xi)\cdot\deg(\eta)} \mathcal{F}_n(\ldots,\eta,\xi,\ldots) \] for $\xi,\eta$ homogeneous, 
\item for each $n \geqslant 1$ and (homogeneous) multivectors $\xi_1$, $\ldots$, $\xi_n \in T_{\poly}^{\downarrow[1]}(M^r)$, we have that (cf.~\cite[\S3.6]{CattaneoKellerTorossian})
\begin{multline}
\label{EqLinfty}
\Id_H(\mathcal{F}_n(\xi_1,\ldots,\xi_n)) - (-)^{n-1} \sum_{i=1}^n (-)^u \mathcal{F}_n(\xi_1,\ldots,d\xi_i,\ldots,\xi_n) 
\\
+\tfrac{1}{2} \sum\nolimits_{\substack{p+q=n\\p,q>0}} \sum\nolimits_{\sigma \in S_{p,q}} (-)^{pn + t} \bigl[\mathcal{F}_p(\xi_{\sigma(1)},\ldots,\xi_{\sigma(p)}),\mathcal{F}_q(\xi_{\sigma(p+1)},\ldots,\xi_{\sigma(n)})\bigr]_G
\\
= (-)^n \sum\nolimits_{i<j} (-)^s \mathcal{F}_{n-1}\bigl([\xi_i,\xi_j],\xi_1,\ldots,\widehat{\xi_i},\ldots,\widehat{\xi_j},\ldots,\xi_n\bigr).
\end{multline}
In the above formula, $\sigma$~runs through the set of $(p,q)$-shuffles, i.e.\ all permutations $\sigma \in S_n$ such that $\sigma(1) <\ldots<\sigma(p)$ and independently $\sigma(p+1) < \ldots < \sigma(n)$\textup{;} the exponents $t$ and $s$ are the numbers of transpositions of odd elements which we count when passing $(t)$ from $(\mathcal{F}_p$, $\mathcal{F}_q$, $\xi_1$, $\ldots$, $\xi_n)$ to $(\mathcal{F}_p$, $\xi_{\sigma(1)}$, $\ldots$, $\xi_{\sigma(p)}$, $\mathcal{F}_q$, $\xi_{\sigma(p+1)}$, $\ldots$, $\xi_{\sigma(n)})$, and $(s)$ from $(\xi_1$, $\ldots$, $\xi_n)$ to $(\xi_i$, $\xi_j$, $\xi_1$, $\ldots$, $\widehat{\xi_1}$, $\ldots$, $\widehat{\xi_j}$, $\ldots$, $\xi_n)$.\footnote{%
The exponent~$u$ is not essential for us now because the differential~$d$ on $T_{\poly}^{\downarrow[1]}(M^r)$ is set equal to zero identically, so that the entire term with~$u$ does not contribute (recall $\mathcal{F}_n$ is linear).}
\end{enumerate}
\end{theor}

\begin{rem}
Let $n\mathrel{{:}{=}} 1$, then equality~\eqref{EqLinfty} in Theorem~\ref{ThMainMK97} is 
\[ \Id_H \circ \mathcal{F}_1 - (-)^{1-1} \cdot (-)^{u=0 \text{ from } (d,\xi_1) \mapsto (d,\xi_1)} F_1 \circ d = 0 \iff \Id_H \circ \mathcal{F}_1 = \mathcal{F}_1 \circ d, \] 
whence $\mathcal{F}_1$ is a morphism of complexes.

\noindent
$\bullet$\quad Let $n \mathrel{{:}{=}} 2$, then for any homogeneous multivectors~$\xi_1$ and~$\xi_2$, 
\[ \mathcal{F}_1\bigl(\schouten{\xi_1,\xi_2}\bigr) - \bigl[\mathcal{F}_1(\xi_1),\mathcal{F}_1(\xi_2)\bigr]_G = \Id_H\bigl(\mathcal{F}_2(\xi_1,\xi_2)\bigr) + \mathcal{F}_2\bigl((d=0)(\xi_1),\xi_2\bigr) + (-)^{\deg \xi_1} \mathcal{F}_2\bigl(\xi_1, (d=0)(\xi_2)\bigr), \]
so that in our case $\mathcal{F}_1$~is ``almost'' a Lie algebra morphism but for the discrepancy which is controlled by the differential of the (value of the) succeeding map~$\mathcal{F}_2$ in the sequence $\bcF = \{\mathcal{F}_n, n \geqslant 1\}$.
Big formula~\eqref{EqLinfty} shows in precisely which sense this is also the case for higher homotopies~$\mathcal{F}_n$, $n \geqslant 2$ in the $L_\infty$-\/morphism~$\bcF$.
Indeed, 
an $L_\infty$-\/morphism is a map between dgLas which, in every term, almost preserves the bracket up to a homotopy~$\Id_H \circ \{ \ldots \}$ provided 
by the next term.
\end{rem}


Even though neither~$\mathcal{F}_1$ nor the entire collection~$\bcF = \{\mathcal{F}_n, n \geqslant 1\}$ is a dgLa morphism, their defining property~\eqref{EqLinfty} guarantees that $\bcF$~gives us a well defined mapping of the Maurer\/--\/Cartan elements (which, we recall, are formal Poisson bi\/-\/vectors and tails~$B$ of associative (non)commutative multiplcations $\star = \mu + B$ on~$A[[\hbar]]$, respectively).

\begin{cor}
\label{CorMCtoMC}
The natural $\hbar$-\/linear extension of~$\bcF$, now acting on the space of formal power series in~$\hbar$ with coefficients in~$T_{\poly}^{\downarrow[1]}(M^r)$ and with zero free term 
by the rule 
\[ \xi \mapsto \sum\nolimits_{n\geqslant 1} \frac{1}{n!} \mathcal{F}_n(\xi,\ldots,\xi), \]
takes the Maurer\/--\/Cartan elements $\tilde{\cP} = \hbar \cP + \bar{o}(\hbar)$ to the Maurer\/--\/Cartan elements $B = \sum_{n\geqslant 1} \frac{1}{n!} \mathcal{F}_n(\tilde{\cP},\ldots,\tilde{\cP}) = \hbar \tilde{\cP} + \bar{o}(\hbar)$.
(Note that the HKR map~$\mathcal{F}_1$, extended by $\hbar$-\/linearity, still is an identity mapping on multivectors, now viewed as special polydifferential operators.)
\end{cor}

In plain terms, for a bivector~$\cP$ itself Poisson, formal Poisson structures $\tilde{\cP} = \hbar \cP + \bar{o}(\hbar)$ satisfying $\schouten{\tilde{\cP}, \tilde{\cP}} = 0$ are mapped by~$\bcF$ 
to the tails $B = \hbar{\cP} + \bar{o}(\hbar)$ such that $\star = \mu + B$ is associative and its leading order deformation term is a given Poisson structure~$\cP$.

\begin{proof}[Proof (of Corollary~\ref{CorMCtoMC})]
Let us presently consider the restricted case when $\tilde{\cP} = \hbar \cP$, without any higher order tail $\bar{o}(\hbar)$.
The Maurer\/--\/Cartan equation in $D_{\poly}^{\downarrow[1]}(M^r) \otimes \hbar \Bbbk[[\hbar]]$ is $[\mu, B]_G + \frac{1}{2}[B,B]_G = 0$, where $B = \sum_{n \geqslant 1} \frac{1}{n!} \mathcal{F}_n(\tilde{\cP},\ldots,\tilde{\cP})$ and we let $\tilde{\cP} = \hbar \cP$, so that $B = \sum_{n \geqslant 1} \frac{\hbar^n}{n!} \mathcal{F}_n(\cP$, $\ldots$, $\cP)$.
Let us plug this formal power series in the l.\/-\/h.s.\ of the above equation.
Equating the coefficients at powers~$\hbar^n$ and multiplying by~$n!$, we obtain the expression
\[ [\mu, \mathcal{F}_n(\cP,\ldots,\cP)]_G + \tfrac{1}{2} \sum\nolimits_{\substack{p+q=n\\p,q>0}} \frac{n!}{p!q!} \bigl[\mathcal{F}_p(\cP,\ldots,\cP),\mathcal{F}_q(\cP,\ldots,\cP)\bigr]_G. 
\]
It is readily seen that now the sum $\sum_{\sigma \in S_{p,q}}$ in~\eqref{EqLinfty} over the set of $(p,q)$-\/shuffles of $n=p+q$ identical copies of an object~$\cP$ just counts the number of ways to pick $p$~copies going first in an ordered string of length~$n$.
To balance the signs, we note at once 
that by item~\ref{LinftySkew} in Theorem~\ref{ThMainMK97}, see above, $\mathcal{F}_p(\ldots,\cP^{(\alpha)},\cP^{(\alpha+1)},\ldots) = +\mathcal{F}_p(\ldots,\cP^{(\alpha+1)},\cP^{(\alpha)},\ldots)$ because bi\/-\/vector's shifted degree is~$+1$, so that no $(p,q)$-\/shuffles of~$(\cP,\ldots,\cP)$ contribute with any sign factor.
The only sign contribution that remains stems from the symbol~$\mathcal{F}_q$ of grading~$1-q$ transported along $p$~copies of odd\/-\/degree bi\/-\/vector~$\cP$; this yields $t = (1-p)\cdot q$ and $(-)^{pn+t} = (-)^{p\cdot(p+q)} \cdot (-)^{(1-q)\cdot p} = (-)^{p\cdot (p+1)} = +$.

The left\/-\/hand side of the Maurer\/--\/Cartan equation~\eqref{EqMC} is, by the above, expressed by the left\/-\/hand side of~\eqref{EqLinfty} which the $L_\infty$-\/morphism $\bcF$ satisfies.
In the right\/-\/hand side of~\eqref{EqLinfty}, we now obtain (with, actually, whatever sign factors) the values of linear mappings~$\mathcal{F}_{n-1}$ at twice the Jacobiator $\schouten{\tilde{\cP},\tilde{\cP}}$ as one of the arguments.
All these values are therefore zero, which implies that the right\/-\/hand side of the Maurer\/--\/Cartan equation~\eqref{EqMC} vanishes, so that the tail~$B$ indeed is a Maurer\/--\/Cartan element in the Hochschild cochain complex (in other words, the star\/-\/product $\star = \mu + B$ is associative).

This completes the proof in the restricted case when $\tilde{\cP} = \hbar \cP$.
Formal power series bi\/-\/vectors $\tilde{\cP} = \hbar \cP + \bar{o}(\hbar)$ refer to the same count of signs as above, yet the calculation of multiplicities at~$\hbar^n$ (for all possible lexicographically ordered $p$- and $q$-\/tuples of $n$~arguments) is an extensive exercise in combinatorics.
\end{proof}

\begin{cor}
\label{CorDiamondF}
Because the right\/-\/hand side of~\eqref{EqMC} in the above reasoning is determined by the right\/-\/hand side of~\eqref{EqLinfty}, we read off an explicit formula of the operator~ $\Diamond$ that solves the factorization problem 
\begin{equation}\tag{\ref{EqDiamondAssoc}}
\Assoc(\star)(\cP)(f,g,h) = \Diamond\bigl(\cP,\schouten{\cP,\cP}\bigr)(f,g,h),
   \qquad f,g,h\in A[[\hbar]]. 
\end{equation}
Indeed, the operator is
\begin{equation}
\label{EqDiamondF}
\Diamond = 2 \cdot \sum\nolimits_{n \geqslant 1} \frac{\hbar^n}{n!} \cdot c_n \cdot \mathcal{F}_{n-1}\bigl(\schouten{\cP,\cP},\cP,\ldots,\cP\bigr).
\end{equation}
\end{cor}

\noindent
But what are the coefficients $c_n\in\BBR$ equal to? Let us find it out. 


\section{Explicit construction of the formality morphism~$\bcF$}
\noindent%
The first explicit formula for the formality morphism~$\bcF$ which we study 
in this paper 
was discovered by Kontsevich in~\cite[\S6.4]{KontsevichFormality}, providing an expansion of every term~$\mathcal{F}_n$ using weighted decorated graphs: 
\[ \bcF = \Bigl\{ \mathcal{F}_n = \sum\nolimits_{m\geqslant 0} \sum\nolimits_{\Gamma \in G_{n,m}} W_\Gamma \cdot \mathcal{U}_\Gamma \Bigr\}. 
\]
Here $\Gamma$~belongs to the set~$G_{n,m}$ of oriented graphs on $n$~internal vertices (i.e.\ arrowtails), $m$~sinks (from which no arrows start), and $2n+m-2 \geqslant 0$ edges, such that at every internal vertex there is an ordering of outgoing edges.
By decorating each edge with a summation index that runs from~$1$ to~$r$, by viewing each edge as a derivation $\partial/\partial x^\alpha$ of the arrowhead vertex content, by placing $n$~multivectors from an ordered tuple of arguments of~$\mathcal{F}_n$ into the respective vertices, now taking the sum over all indices of the resulting products of the content of vertices, and skew\/-\/symmetrizing over the $n$-\/tuple of (shifted-)\/graded multivectors, we realize each graph at hand as a polydifferential operator $T_{\poly}^{\downarrow[1]}(M^r)^{\otimes n} \to D_{\poly}^{\downarrow[1]}(M^r)$ whose arguments are multivectors.
Note that the \emph{value} $\mathcal{F}_n(\xi_1,\ldots,\xi_n)$ itself is, by construction, a differential operator w.r.t.\ the contents of sinks of the graph~$\Gamma$.
All of this is discussed in detail in~\cite{MKParisECM,
Ascona96,KontsevichFormality} or~\cite{f16,sqs15,cpp}.

The formula for the harmonic weights~$W_\Gamma \in \mathbb{R}$ is given in~\cite[\S6.2]{KontsevichFormality}; it is \[ W_\Gamma = \Bigg(\prod_{k=1}^n \frac{1}{\#\!\Star(k)!}\Bigg) \cdot \frac{1}{(2\pi)^{2n+m-2}} \int_{\bar{C}^+_{n,m}} \bigwedge \limits_{e \in E_\Gamma} d\phi_e, \]
where $\#\Star(k)$ is the number of edges \emph{star}ting from vertex $k$, $d\varphi_e$ is the ``harmonic angle'' differential $1$-form associated to the edge $e$, and the integration domain $\bar{C}^+_{n,m}$ is the connected component of $\bar{C}_{n,m}$ which is the closure of configurations where points $q_j$, $1 \leqslant j \leqslant m$ on $\mathbb{R}$ are placed in increasing order: $q_1 < \cdots < q_m$.
For convenience, let us also define \[ w_\Gamma = \bigg(\prod_{k=1}^n \#\!\Star(k)! \bigg) \cdot  W_\Gamma. \]
The convenience is that by summing over labelled graphs $\Gamma$, we actually sum over the equivalence classes $[\Gamma]$ (i.e. over unlabeled graphs) with multiplicities $(w_\Gamma/W_\Gamma)\cdot n! / \#\!\Aut(\Gamma)$.
The division by the volume $\#\!\Aut(\Gamma)$ of the symmetry group eliminates the repetitions of graphs which differ only by a labeling of vertices but, modulo such, do not differ by the labeling of ordered edge tuples (issued from the vertices which are matched by a symmetry).

Let us remember that the integrand in the formula of~$W_\Gamma$ 
is defined in terms of the harmonic propagator; other propagators (e.g.\ logarithmic, or other members of the family interpolating between harmonic and logarithmic~\cite{RossiWillwacher2014}) would give other formality morphisms.
A path integral realization of the $\star$-\/product itself and of the components~$\mathcal{F}_n$ in the formality morphism is proposed in~\cite{CattaneoFelder2000}.

To calculate the graph weights~$W_\Gamma$ in practice, we employ methods which were outlined in~\cite{cpp}, as well as~\cite[App.~E]{FelderWillwacher2008} (about the cyclic weight relations), and~\cite{BanksPanzerPym1812} that puts those real values in the context of Riemann multiple zeta functions and polylogarithms.\footnote{It is the values $w_\Gamma$ instead of $W_\Gamma$ which are calculated by software~\cite{BanksPanzerPym1812}.}
Examples of such decorated oriented graphs~$\Gamma$ and their weights~$W_\Gamma$ will be given in the next section.

\subsection{Sum over equivalence classes}

The sum in Kontsevich's formula is over \emph{labeled} graphs: internal vertices are numbered from $1$ to $n$, and the edges starting from each internal vertex $k$ are numbered from $1$ to $\#\!\Star(k)$.
Under a re-labeling $\sigma: \Gamma \mapsto \Gamma^\sigma$ of internal vertices and edges it is seen from the definitions that the operator $\mathcal{U}_\Gamma$ and the weight $W_\Gamma$ enjoy the same skew\/-\/symmetry property (as remarked in \cite[\S6.5]{KontsevichFormality}), whence $W_\Gamma \cdot \mathcal{U}_\Gamma = W_{\Gamma^\sigma} \cdot \mathcal{U}_{\Gamma^\sigma}$.
It follows that the sum over labeled graphs can be replaced by a sum over equivalence classes $[\Gamma]$ of graphs, modulo labeling of internal vertices and edges.
For this it remains to count the size of an equivalence class: the edges can be labeled in $\prod_{k=1}^n \#\!\Star(k)!$ ways, while the $n$ internal vertices can be labeled in $n!/\#\!\Aut(\Gamma)$ ways.
\begin{example}
The double wedge on two ground vertices has only \emph{one} possible labeling of vertices, due to the automorphism that interchanges the wedges.
\end{example}
We denote by $M_\Gamma = \big(\prod_{k=1}^n \#\!\Star(k)!\big) \cdot n!/\#\!\Aut(\Gamma)$ the \emph{multiplicity} of the graph $\Gamma$, and let $\bar{G}_{n,m}$ be the set of equivalence classes $[\Gamma]$ modulo labeling of $\Gamma \in G_{n,m}$.
The formula for the formality morphism can then be rewritten as
\[ \bcF = \Bigl\{ \mathcal{F}_n = \sum\nolimits_{m\geqslant 0} \sum\nolimits_{[\Gamma] \in \bar{G}_{n,m}} M_\Gamma \cdot W_\Gamma \cdot \mathcal{U}_\Gamma \Bigr\};
\]
here the $\Gamma$ in $M_\Gamma \cdot W_{\Gamma} \cdot \mathcal{U}_\Gamma$ is \emph{any} representative of $[\Gamma]$.
Any ambiguity in signs (due to the choice of representative) in the latter two factors is cancelled in their product.
Note that the factor $\big(\prod_{k=1}^n \#\!\Star(k)!\big)$ in $M_\Gamma$ kills the corresponding factor in $W_\Gamma$, as remarked in \cite[\S6.5]{KontsevichFormality}.

\subsection{The coefficient of a graph in the $\star$-product}
\label{SecCoeffStar}

The $\star$-product associated to a Poisson structure $\cP$ is given by Corollary~\ref{CorMCtoMC}: \[\star = \mu + \sum_{n \geqslant 1} \frac{\hbar^n}{n!} \mathcal{F}_n(\cP,\ldots,\cP) = \mu + \sum_{n \geqslant 1} \frac{\hbar^n}{n!} \sum_{[\Gamma] \in \bar{G}_{n,2}} M_\Gamma \cdot W_\Gamma \cdot \mathcal{U}_\Gamma(\cP,\ldots,\cP). \]
For a graph $\Gamma \in G_{n,2}$ such that each internal vertex has two outgoing edges (these are the only graphs that contribute, because we insert bi-vectors) we have $M_\Gamma = 2^n \cdot n!/\#\!\Aut(\Gamma)$.
In total, the coefficient of $\mathcal{U}_\Gamma(\cP,\ldots,\cP)$ at $\hbar^n$ is $2^n / \#\!\Aut(\Gamma) \cdot W_\Gamma = w_\Gamma/\#\!\Aut(\Gamma)$.
The skew-symmetrization \emph{without prefactor} of bi-vector coefficients in $\mathcal{U}_\Gamma(\cP,\ldots,\cP)$ provides an extra factor $2^n$.
\begin{example}[at $\hbar^1$]
The coefficient of the wedge graph is $1/2$ and the operator is $2\cP$, hence we recover $\cP$.
\end{example}

\subsection{The coefficient of a Leibniz graph in the associator}
\label{SecCoeffLeibniz}

The factorizing operator $\Diamond$ for $\Assoc(\star)$ is given by Corollary \ref{CorDiamondF}:
\begin{align*}
\Diamond &= 2 \cdot \sum_{n\geqslant 1} \frac{\hbar^n}{n!} \cdot c_n \cdot \mathcal{F}_{n-1}\bigl(\schouten{\cP,\cP},\cP,\ldots,\cP\bigr) \\
& = 2 \cdot \sum_{n\geqslant 1} \frac{\hbar^n}{n!} \cdot c_n \cdot \sum_{[\Gamma] \in \bar{G}_{n-1,3}} M_\Gamma \cdot W_\Gamma \cdot \mathcal{U}_\Gamma\bigl(\schouten{\cP,\cP},\cP,\ldots,\cP\bigr).
\end{align*}
For a graph $\Gamma \in G_{n-1,3}$ where one internal vertex has three outgoing edges and the rest have two, we have $M_\Gamma = 3! \cdot 2^{n-2} \cdot (n-1)! / \#\!\Aut(\Gamma)$.
In total, the coefficient of $\mathcal{U}_\Gamma(\schouten{\cP,\cP}, \cP, \ldots, \cP)$ at $\hbar^n$ is
\[ \biggl[2 \cdot \frac{1}{n!} \cdot c_n \cdot 3! \cdot 2^{n-2} \cdot (n-1)!\biggr] \cdot \frac{W_\Gamma}{\#\!\Aut(\Gamma)}  = \biggl[2 \cdot \frac{c_n}{n}\biggl] \cdot \frac{w_\Gamma}{\#\!\Aut(\Gamma)} \]
The skew-symmetrization \emph{without prefactor} of bi- and tri-vector coefficients in the operator $\mathcal{U}_\Gamma(\schouten{\cP,\cP},\cP,\ldots,\cP)$ provides an extra factor $3! \cdot 2^{n-2}$.
\begin{example}[at $\hbar^2$]
\label{ExTripodCoeff}
The coefficient of the tripod graph is $c_2 \cdot \frac{1}{3!}$ and the operator is $3! \cdot \schouten{\cP,\cP}$, hence we recover $c_2 \schouten{\cP,\cP} = \tfrac{2}{3} \Jac(\cP)$.
(The right-hand side is known from the associator, e.g. from~\cite{sqs15}.)
This yields $c_2=1/3$.
In addition, we see that the HKR map $\mathcal{F}_1$ acts here by the identity on $\schouten{\cP,\cP}$.
\end{example}
In the next section, we shall find that at $\hbar^n$, the coefficients of our Leibniz graphs (with $\Jac(\cP)$ inserted instead of $\schouten{\cP,\cP}$) are 
\[ \frac{\schouten{P,P}}{\Jac(\cP)} \cdot \biggl[ 3! \cdot 2^{n-2} \biggr] \cdot \biggl[ 2 \cdot \frac{c_n}{n}\biggr] \cdot \frac{w_\Gamma}{\#\!\Aut(\Gamma)} = 2^n \cdot \frac{w_\Gamma}{\#\!\Aut(\Gamma)}, \qquad \text{ so } \qquad 3! \cdot 2^{n} \cdot \frac{c_n}{n} = 2^n.\]
%
We deduce that $c_n = n/3! = n/6$ in all our experiments.

\begin{conjecture}
For all $n\geqslant 2$, the coefficients in~\eqref{EqDiamondF} are $c_n=n/3!=n/6$ (hence, the coefficients of markers $\Gamma$ for equivalence classes $[\Gamma]$ of the Leibniz graphs in \eqref{EqDiamondF} are $2^n \cdot w_\Gamma / \#\!\Aut(\Gamma)$), although it still remains to be explained how exactly this follows from the $L_\infty$ condition~\eqref{EqLinfty}.
\end{conjecture}


\section{Examples}
\noindent%
Let $\cP$ be a Poisson bi\/-\/vector on an affine manifold~$M^r$.
We inspect the asssociativity of the star\/-\/product $\star = \mu + \sum_{n \geqslant 1} \frac{\hbar^n}{n!}\mathcal{F}_n(\cP$, $\ldots$, $\cP)$ given by Corollary~\ref{CorMCtoMC} by illustrating the work of the factorization mechanism from Corollary~\ref{CorDiamondF}.
The powers of deformation parameter~$\hbar$ provide a natural filtration $\hbar^2\cdot\mathsf{A}^{(2)}+
\hbar^3\cdot\mathsf{A}^{(3)}+\hbar^4\cdot\mathsf{A}^{(4)}+\bar{o}(\hbar^4)$
so that we verify the vanishing of $\Assoc(\star)(\cP)(\cdot,\cdot,\cdot) \mod \bar{o}(\hbar^4)$ for~$\star$ mod~$\bar{o}(\hbar^4)$ order by order.

At~$\hbar^0$ there is nothing to do (indeed, the usual multiplication is associative).
All contribution to the associator of~$\star$ at~$\hbar^1$ cancels out because the leading deformation term~$\hbar \cP$ in the star\/-\/product $\star = \mu + \hbar P + \bar{o}(\hbar)$ is a bi\/-\/derivation.
The order~$\hbar^2$ was discussed in Example~\ref{ExTripodCoeff} in \S\ref{SecCoeffLeibniz}.

\begin{rem}\label{RemSplitJac}
In all our reasoning at any order~$\hbar^{n\geqslant 2}$, the Jacobiator in Leibniz graphs is expanded (w.r.t.\ the three cyclic permutations of its arguments) into the Kontsevich graphs, built of wedges, in such a way that the internal edge, connecting two Poisson bi\/-\/vectors in~$\Jac(\cP)$, is proclaimed Left by construction. 
Specifically, 
the algorithm 
to expand each Leibniz graphs is as follows:
\begin{enumerate}
\item Split the trivalent vertex with ordered targets $(a,b,c)$ 
into two wedges: the first wedge stands on~$a$ and~$b$ (in that order), and the second wedge stands on the first wedge\/-\/top and~$c$ (in that order), so that the internal edge of the Jacobiator is marked Left, preceding the Right edge towards~$c$.
\item Re-direct the edges (if any) which had 
the tri\/-\/valent vertex as their target, to one of the wedge\/-\/tops; take the sum over all possible combinations (this is the iterated Leibniz rule).
\item Take the sum over cyclic permutations of the targets of the edges which (initially) have $(a,b,c)$ as their targets (this is the expansion of the Jacobiator).
\end{enumerate}
\end{rem}

\subsection{The order~$\hbar^3$}
\label{ExLeibniz3}
To factorize the next order expansion of the associator, $\Assoc(\star)(\cP)$ mod~$\bar{o}(\hbar^3) = \hbar^2\cdot\mathsf{A}^{(2)}+
\hbar^3\cdot\mathsf{A}^{(3)}+\bar{o}(\hbar^3)$, at~$\hbar^3$ in the operator~$\Diamond$ in the right\/-\/hand side of~\eqref{EqDiamondAssoc}, we use graphs on $n-1 = 2$ vertices, 
$m=3$ sinks, and $2(n-1)+m-2 = 5$ edges. 

At~$\hbar^3$, two internal vertices in the Leibniz graphs in the r.\/-\/h.s.\ of factorization~\eqref{EqDiamondAssoc} are manifestly different: one vertex, containg the bi\/-\/vector~$\cP$, is a source of two outgoing edges, and the other, with $\schouten{\cP,\cP}$, of three. Therefore, 
the automorphism groups of such Leibniz graphs (under relabellings of internal vertices of the same valency but with the sinks fixed) can only be trivial, i.e.\ one\/-\/element. (This will not necessarily be the case of Leibniz graphs on $(n-2)+1$ internal vertices at~$\hbar^{\geqslant4}$: compare Examples~\ref{ExA331} vs~\ref{ExA322} on p.~\pageref{ExA322} below, where the weight of a graph is divided further by the size of its automorphism group.)

The coefficient of~$\hbar^3$ in the factorizing operator~$\Diamond$, 
\[ \operatorname{coeff}(\Diamond,\hbar^3) = 2 \cdot \frac{1}{3!} \cdot c_3 \cdot \sum_{[\Gamma] \in \bar{G}_{2,3}} M_\Gamma \cdot W_\Gamma \cdot \mathcal{U}_\Gamma\bigl(\schouten{\cP,\cP},\cP,\ldots,\cP\bigr), \]
expands into a sum of~$\leqslant 24$ 
admissible oriented graphs.
Indeed, there are six 
essentially different oriented graph topologies, filtered by the number of sinks on which the tri\/-\/vector $\schouten{\cP,\cP}$ and bi\/-\/vector~$\cP$ stand; the ordering of sinks in the associator then yields $3+3+3\times 2+3\times 2+3 = 24$ oriented graphs. 
(None of them is a zero orgraph.)
As we recall from~\cite{sqs15}, only thirteen 
of them actually occur with nonzero coefficients in the term $\mathsf{A}^{(3)}\sim\hbar^3$ in $\Assoc(\star)(\cP))$, the remaining 
eleven have zero weights.\footnote{Yet, these seemingly `unnecessary' graphs can contribute to the cyclic weight relations (see~\cite[App.~E]{FelderWillwacher2008}): zero values of some of such graph weights can simplify the system of linear relations between nonzero weights.}
The 
weights of $15$~relevant 
oriented Leibniz graphs from~\cite{sqs15} are listed 
in Table
~\ref{FigLeibniz3}.%
\footnote{%
To get the values, one uses the software~\cite{BanksPanzerPym1812} by Banks\/--\/Panzer\/--\/Pym or, independently, 
exact symbolic or approximate numeric methods from~\cite{cpp}, also taking into account the cyclic weight relations from~\cite[App.~E]{FelderWillwacher2008}.}%
\begin{table}[htb]
\caption{Weights $w_\Gamma$ of oriented Leibniz graphs $\Gamma$
in~$\operatorname{coeff}(\Diamond,\hbar^3)$.}\label{FigLeibniz3}
\centerline{
\renewcommand{\arraystretch}{1.3}
\begin{tabular}{cccc|cccc|cccc}
\hline
$(S_{\!\!f})_{221}$&=&[$\mathsf{01};\mathsf{012}$]&$\frac{1}{12}$ &
$(S_{\!g})_{122}$&=&[$\mathsf{12};\mathsf{012}$]&$\frac{1}{12}$ &
$(S_{\!h})_{212}$&=&[$\mathsf{20};\mathsf{012}$]&$\frac{ - 1}{12}$ 
\\{}
$(I_f)_{112}$&=&[$\mathsf{02};\mathsf{312}$]&$\frac{1}{48}$ &
$(I_g)_{112}$&=&[$\mathsf{12};\mathsf{032}$]&$\frac{1}{48}$ &
$(S_{\!h})_{112}$&=&[$\mathsf{24};\mathsf{012}$]&$\frac{ - 1}{24}$ 
\\{}
$(S_{\!\!f})_{211}$&=&[$\mathsf{04};\mathsf{012}$]&$\frac{1}{24}$ &
$(I_g)_{211}$&=&[$\mathsf{10};\mathsf{032}$]&$\frac{ - 1}{48}$ &
$(I_h)_{211}$&=&[$\mathsf{20};\mathsf{013}$]&$\frac{ - 1}{48}$ 
\\{}
$(I_f)_{111}$&=&[$\mathsf{04};\mathsf{312}$]&$\frac{1}{48}$ &
$(I_h)_{111}$&=&[$\mathsf{24};\mathsf{013}$]&$\frac{ - 1}{48}$ &
$(I_g)_{111}$&=&[$\mathsf{14};\mathsf{032}$]&$0$ 
\\{}
$(S_{\!g})_{111}$&=&[$\mathsf{14};\mathsf{012}$]&$0$ &
$(I_f)_{121}$&=&[$\mathsf{01};\mathsf{312}$]&$\frac{1}{24}$ &
$(I_h)_{121}$&=&[$\mathsf{21};\mathsf{013}$]&$\frac{ - 1}{24}$ 
\\
\hline
\end{tabular}%
}
\end{table}

Here we let by definition 
\[
I_f \mathrel{{:}{=}} 
\dd_j\bigl( \Jac(\cP)(\cP^{ij}, g, h) \bigr)\,\dd_i f =
\text{\raisebox{-20pt}{
\unitlength=0.70mm
\linethickness{0.4pt}
\begin{picture}(26.00,16.33)
\put(13.00,5.00){\line(1,0){13.00}}
\put(1,-3){\line(3,2){12.00}}
\put(2.00,5.00){\circle*{1.33}}
\put(13.00,5.00){\circle*{1.33}}
\put(24.00,5.00){\circle*{1.33}}
\put(7.33,11.33){\circle*{1.33}}
\put(7.33,11.33){\vector(1,-1){5.5}}
\put(7.33,11.33){\vector(-1,-1){5.5}}
\put(13,17){\circle*{1.33}}
\put(13,17){\vector(1,-1){11.2}}
\put(13,17){\vector(-1,-1){5.1}}
\put(13,10.3){\oval(30,20)}
\put(2,5.0){\vector(0,-1){7.1}}
\qbezier(2,5)(-8,-2)(-2,9)
\put(-2.6,8.0){\vector(1,2){0.5}}
\put(2,-2){\circle*{1.33}}
\put(-7,3){\scriptsize $j$}
\end{picture}
}}
\ 
{-}
\text{\raisebox{-20pt}{
\unitlength=0.70mm
\linethickness{0.4pt}
\begin{picture}(26.00,16.33)
\put(13.00,5.00){\line(1,0){13.00}}
\put(1,-3){\line(3,2){12.00}}
\put(2.00,5.00){\circle*{1.33}}
\put(13.00,5.00){\circle*{1.33}}
\put(24.00,5.00){\circle*{1.33}}
\put(13,11.33){\circle*{1.33}}
\put(13,11.33){\vector(2,-1){10.8}}
\put(13,11.33){\vector(-2,-1){10.8}}
\put(18.5,17){\circle*{1.33}}
\put(18.5,17){\vector(-1,-1){5.2}}
\put(18.5,17){\vector(-1,-2){5.6}}
\put(13,10.3){\oval(30,20)}
\put(2,5.0){\vector(0,-1){7.1}}
\qbezier(2,5)(-8,-2)(-2,9)
\put(-2.6,8.0){\vector(1,2){0.5}}
\put(2,-2){\circle*{1.33}}
\put(17,12){\tiny $R$}
\put(-7,3){\scriptsize $j$}
\end{picture}
}}
\ 
{-}
\text{\raisebox{-20pt}{
\unitlength=0.70mm
\linethickness{0.4pt}
\begin{picture}(26.00,16.33)
\put(13.00,5.00){\line(1,0){13.00}}
\put(1,-3){\line(3,2){12.00}}
\put(2.00,5.00){\circle*{1.33}}
\put(13.00,5.00){\circle*{1.33}}
\put(24.00,5.00){\circle*{1.33}}
\put(18.33,11.33){\circle*{1.33}}
\put(18.33,11.33){\vector(1,-1){5.5}}
\put(18.33,11.33){\vector(-1,-1){5.5}}
\put(13,17){\circle*{1.33}}
\put(13,17){\vector(-1,-1){11.2}}
\put(13,17){\vector(1,-1){5.1}}
\put(13,10.3){\oval(30,20)}
\put(2,5.0){\vector(0,-1){7.1}}
\qbezier(2,5)(-8,-2)(-2,9)
\put(-2.6,8.0){\vector(1,2){0.5}}
\put(2,-2){\circle*{1.33}}
\put(-7,3){\scriptsize $j$}
\end{picture}
}} = 0.
\]
Likewise, 
$I_g \mathrel{{:}{=}} \partial_j\bigl(\Jac(\cP)(f,\cP^{ij},h)\bigr) \cdot\partial_i g$ and $I_h \mathrel{{:}{=}} \partial_j\bigl(\Jac(\cP)(f,g,\cP^{ij})\cdot\partial_i h$, respectively.\footnote{%
In~\cite{sqs15}, the indices~$i$ and~$j$ were interchanged in the definitions of both~$I_g$ and~$I_h$ (compare the expression of~$I_f$);
that typo is now corrected in the above formulae.%
}

We also set
\[
S_f \mathrel{{:}{=}}
\cP^{ij} \dd_j \Jac(\cP)(\dd_i f, g, h) =
\text{\raisebox{-16pt}{
\unitlength=0.70mm
\linethickness{0.4pt}
\begin{picture}(26.00,16.33)
\put(-6, 13){\scriptsize $i$}
\put(0.00,5.00){\line(1,0){26.00}}
\put(2.00,5.00){\circle*{1.33}}
\put(13.00,5.00){\circle*{1.33}}
\put(24.00,5.00){\circle*{1.33}}
\put(7.33,11.33){\circle*{1.33}}
\put(7.33,11.33){\vector(1,-1){5.5}}
\put(7.33,11.33){\vector(-1,-1){5.5}}
\put(13,17){\circle*{1.33}}
\put(13,17){\vector(1,-1){11.2}}
\put(13,17){\vector(-1,-1){5.1}}
\put(13,10.3){\oval(30,20)}
\put(-5.7,19.0){\circle*{1.33}}
\put(-5.7,19.0){\vector(1,-2){6.8}}
\put(-5.7,19.0){\vector(2,-1){5.1}}
\end{picture}
}}
\ 
{-}
\text{\raisebox{-16pt}{
\unitlength=0.70mm
\linethickness{0.4pt}
\begin{picture}(26.00,16.33)
\put(-6, 13){\scriptsize $i$}
\put(0.00,5.00){\line(1,0){26.00}}
\put(2.00,5.00){\circle*{1.33}}
\put(13.00,5.00){\circle*{1.33}}
\put(24.00,5.00){\circle*{1.33}}
\put(13,11.33){\circle*{1.33}}
\put(13,11.33){\vector(2,-1){10.8}}
\put(13,11.33){\vector(-2,-1){10.8}}
\put(18.5,17){\circle*{1.33}}
\put(18.5,17){\vector(-1,-1){5.2}}
\put(18.5,17){\vector(-1,-2){5.6}}
\put(13,10.3){\oval(30,20)}
\put(-5.7,19.0){\circle*{1.33}}
\put(-5.7,19.0){\vector(1,-2){6.8}}
\put(-5.7,19.0){\vector(2,-1){5.1}}
\put(13,15){\tiny $L$}
\put(17,12){\tiny $R$}
\end{picture}
}}
\ 
{-}
\text{\raisebox{-16pt}{
\unitlength=0.70mm
\linethickness{0.4pt}
\begin{picture}(26.00,16.33)
\put(-6, 13){\scriptsize $i$}
\put(0.00,5.00){\line(1,0){26.00}}
\put(2.00,5.00){\circle*{1.33}}
\put(13.00,5.00){\circle*{1.33}}
\put(24.00,5.00){\circle*{1.33}}
\put(18.33,11.33){\circle*{1.33}}
\put(18.33,11.33){\vector(1,-1){5.5}}
\put(18.33,11.33){\vector(-1,-1){5.5}}
\put(13,17){\circle*{1.33}}
\put(13,17){\vector(-1,-1){11.2}}
\put(13,17){\vector(1,-1){5.1}}
\put(13,10.3){\oval(30,20)}
\put(-5.7,19.0){\circle*{1.33}}
\put(-5.7,19.0){\vector(1,-2){6.8}}
\put(-5.7,19.0){\vector(2,-1){5.1}}
\end{picture}
}} = 0.
\]
Similarly, we let $S_g \mathrel{{:}{=}} \cP^{ij} \dd_j 
\Jac(\cP)(f, \dd_i 
g, h) 
= 0$ and $S_h \mathrel{{:}{=}} \cP^{ij} \dd_j 
\Jac(\cP)(f, g, \dd_i 
h) = 0$.
Note that after all the Leibniz rules are reworked, each of the six graphs~$I_f$, $\ldots$, $S_h$ --\,with the Jacobiator $\Jac(\cP)=\tfrac{1}{2}\lshad\cP,\cP\rshad$ at the tri\/-\/valent vertex\,-- splits into several homogeneous components, like $(I_f)_{111}$ or~$(S_h)_{212}$; taken alone, each of the components encodes a zero polydifferential operator of respective orders.

\begin{claim}\label{ClaimWeightsInSQS15}
Multiplied by a common factor~$\bigl(\lshad\cP,\cP\rshad/\Jac(\cP)\bigr)\cdot 2^{k-1} = 2\cdot 4=8$, the Leibniz graph weights from Table
~\ref{FigLeibniz3} at~$\hbar^3$ 
fully reproduce the factorization 
which was found in the main Claim in~\cite{sqs15}, namely\textup{:}
\begin{align*}
\mathsf{A}^{(3)}_{221} &=\tfrac{2}{3} (S_f)_{221}, \qquad
\mathsf{A}^{(3)}_{122} = \tfrac{2}{3} (S_g)_{122}, \qquad
\mathsf{A}^{(3)}_{212} = -\tfrac{2}{3} (S_h)_{212}, \\
\mathsf{A}^{(3)}_{111} &= \tfrac{1}{6} (I_f - I_h)_{111}, \qquad
\mathsf{A}^{(3)}_{112} = \bigl(\tfrac{1}{6}I_f + \tfrac{1}{6}I_g - \tfrac{1}{3}S_h\bigr)_{112}, \\
\mathsf{A}^{(3)}_{121} &= \tfrac{1}{3} (I_f - I_h)_{121}, \qquad
\mathsf{A}^{(3)}_{211} = \bigl(\tfrac{1}{3}S_f - \tfrac{1}{6}I_g - \tfrac{1}{6}I_h\bigr)_{211}.
\end{align*}
\end{claim}
\noindent%
Otherwise speaking, the sum of these Leibniz oriented graphs with these weights (times $2\cdot4=8$),
when expanded into the sum of $39$~weighted Kontsevich graphs (built only of wedges), equals identically the $\hbar^3$-\/proportional term in the associator~$\Assoc(\star)(\cP)(f,g,h)$.

\begin{proof}[
Proof scheme]
The encodings of weighted Kontsevich\/-\/graph expansions of the homogeneous components of the weighted Leibniz graphs $I_f$, $\ldots$, $S_h$, which show up in the associator at~$\hbar^3$ and which are processed according to the algorithm in Remark~\ref{RemSplitJac}, are listed in Appendix~\ref{AppIfShIntoK}.
Reducing that collection modulo skew symmetry at internal vertices, we reproduce, as desired, the entire term~$\mathsf{A}^{(3)}$ in the expansion~$\hbar^2\cdot\mathsf{A}^{(2)}+
\hbar^3\cdot\mathsf{A}^{(3)}+\bar{o}(\hbar^3)$ of the associator
$\Assoc(\star)(\cP)$ mod~$\bar{o}(\hbar^3)$.
\end{proof}

Three examples, corresponding to the leftmost column of equalities in Claim~\ref{ClaimWeightsInSQS15}, illustrate this scheme at order~$\hbar^3$. The three cases differ in that for~$\mathsf{A}^{(3)}_{221}$ in Example~\ref{ExA221}, there is just one Leibniz graph without any arrows acting on the Jacobiator vertex.
In the other Example~\ref{ExA121} for~$\mathsf{A}^{(3)}_{121}$, there are two Leibniz graphs still without Leibniz\/-\/rule actions on the Jacobiators in them, so that we aim to show how similar terms are collected.\footnote{%
To collect and compare the Kontsevich orgraphs (built of wedges, i.e.\ ordered edge pairs issued from internal vertices), we can bring every such graph to its normal form, that is, represent it using the \emph{minimal} base\/-\/($\#$\,sinks $+$ $\#$\,internal vertices) number, encoding the graph as the list of ordered pairs of target vertices, by running over all the relabellings of internal vertices. (The labelling of ordered sinks is always $\mathsf{0}\prec\mathsf{1}\prec\ldots\prec\mathsf{m-1}$.)}
Finally, in Example~\ref{ExA111} about~$\mathsf{A}^{(3)}_{111}$ there are two Leibniz graphs with one Leibniz rule action per either graph: an arrow targets the two internal vertices in the Jacobiator.

\begin{example}\label{ExA221}
Take the Leibniz graph $(S_{\!\!f})_{221} = [\mathsf{01};\mathsf{012}]$. 
Its weight is $1/12$.
Multiplying the Leibniz graph by $8$ times its weight and expanding the Jacobiator (there are no Leibniz rules to expand) yields the sum of three Kontsevich graphs: $\frac{2}{3}\big([\mathsf{01};\mathsf{01};\mathsf{42}] + [\mathsf{01};\mathsf{12};\mathsf{40}] + [\mathsf{01};\mathsf{20};\mathsf{41}]\big)$.
This is identically equal to the differential order $(2,2,1)$ homogeneous part $\mathsf{A}^{(3)}_{221}$ of $\Assoc(\star)(\cP)$ at~$\hbar^3$.
For instance, these terms are listed in~\cite[App.~D]{cpp}.
\end{example}

\begin{example}\label{ExA121}
Take the Leibniz graphs $(I_f)_{121} = [\mathsf{01};\mathsf{312}]$ and $(I_h)_{121} = [\mathsf{21};\mathsf{013}]$.
Their weights are $1/24$ and $-1/24$, respectively; multiply them by~$8$.
Expanding the Jacobiator in the linear combination $\tfrac{1}{3} (I_f - I_h)_{121}$ yields the sum of Kontsevich graphs $\frac{1}{3}\big([\mathsf{01};\mathsf{31};\mathsf{42}] + [\mathsf{01};\mathsf{12};\mathsf{43}] + [\mathsf{01};\mathsf{23};\mathsf{41}] - [\mathsf{21};\mathsf{01};\mathsf{43}] - [\mathsf{21};\mathsf{13};\mathsf{40}] - [\mathsf{21};\mathsf{30};\mathsf{41}]\big)$.
The two Leibniz graphs have a Kontsevich graph in common: $[\mathsf{01};\mathsf{12};\mathsf{43}] = [\mathsf{21};\mathsf{01};\mathsf{43}]$ (recall that internal vertex labels can be permuted at no cost and the swap $L \rightleftarrows R$ at a wedge costs a minus sign).
This gives one cancellation; the remaining four terms equal $\mathsf{A}^{(3)}_{121}$ as listed in~\cite[App.~D]{cpp}.
\end{example}

\begin{example}\label{ExA111}
Take the Leibniz graphs $(I_f)_{111}=[\mathsf{04};\mathsf{312}]$ and $(I_h)_{111}=[\mathsf{24};\mathsf{013}]$.
Their weights are $1/48$ and $-1/48$, respectively; multiply them by~$8$.
Expanding the Jacobiator and the Leibniz rule in the linear combination $\tfrac{1}{6} (I_f - I_h)_{111}$ yields the sum of Kontsevich graphs:
\begin{multline*}
\tfrac{1}{6}\big(
[\mathsf{04};\mathsf{31};\mathsf{42}]+
[\mathsf{04};\mathsf{12};\mathsf{43}]+
[\mathsf{04};\mathsf{23};\mathsf{41}]+
[\mathsf{05};\mathsf{31};\mathsf{42}]+
[\mathsf{05};\mathsf{12};\mathsf{43}]+
[\mathsf{05};\mathsf{23};\mathsf{41}]\\
{}-[\mathsf{24};\mathsf{01};\mathsf{43}]
-[\mathsf{24};\mathsf{13};\mathsf{40}]
-[\mathsf{24};\mathsf{30};\mathsf{41}]
-[\mathsf{25};\mathsf{01};\mathsf{43}]
-[\mathsf{25};\mathsf{13};\mathsf{40}]
-[\mathsf{25};\mathsf{30};\mathsf{41}]
\big).
\end{multline*}
Two pairs of graphs cancel; namely $[\mathsf{05};\mathsf{31};\mathsf{42}] = [\mathsf{25};\mathsf{30};\mathsf{41}]$ and $[\mathsf{05};\mathsf{23};\mathsf{41}] = [\mathsf{25};\mathsf{13};\mathsf{40}]$.
The remaining eight terms equal $\mathsf{A}^{(3)}_{111}$ as listed in~\cite[App.~D]{cpp}.
\end{example}

\subsection{The order~$\hbar^4$}
\label{ExLeibniz4}
Let us proceed with the term~$\mathsf{A}^{(4)}$ at~$\hbar^4$ in the associator $\Assoc(\star)(\cP)(\cdot,\cdot,\cdot)$ mod~$\bar{o}(\hbar^4)$.
The numbers of Kontsevich oriented graphs in the star\/-\/product expansion grow as fast as
\begin{multline*}
\star = \hbar^0 \cdot (\#\text{graphs} = 1) + \hbar^1 \cdot (\# = 1) + \hbar^2 \cdot (\# = 4) + \hbar^3 \cdot (\# = 13) + \hbar^4 \cdot (\# = 247) + \\ 
+ \hbar^5 \cdot (\# = 2356) + \hbar^6 \cdot (\# = 66041) + \bar{o}(\hbar^6);
\end{multline*}
here we report the count of all nonzero\/-\/weight Kontsevich oriented graphs.
Counting them modulo automorphisms (which may also swap the sinks), Banks, Panzer, and Pym obtain the numbers $(\hbar^0: 1$, $\hbar^1: 1$, $\hbar^2: 3$, $\hbar^3: 8$, $\hbar^4: 133$, $\hbar^5: 1209$, $\hbar^6: 33268)$.
This shows that at orders~$\hbar^{k \geqslant 4}$, the use of graph\/-\/processing software is indispensible in the task of verifying factorization~\eqref{EqDiamondAssoc} using weighted graph expansion~\eqref{EqDiamondF} of the operator~$\Diamond$.


Specifically, the number of Kontsevich oriented graphs at~$\hbar^k$ in the left\/-\/hand side of the factorization problem $\Assoc(\star)(\cP)(\cdot,\cdot,\cdot) = \Diamond\bigl(\cP$, $\schouten{\cP,\cP}\bigr)(\cdot,\cdot,\cdot)$, and the number of Leibniz graphs which assemble with nonzero coefficients to a solution~$\Diamond$ in the right\/-\/hand side is presented in Table~\ref{TabNumGraphsInAssoc}.
\begin{table}[htb]
\caption{
Number of graphs in either side of the factorization.}%
\label{TabNumGraphsInAssoc}
\begin{tabular}{l r r r r r r}
\hline
$k$ & 2 & 3 & 4 & 5 & 6 & 7 \\
\hline
LHS: $\#$ K. orgraphs & 3\:(Jac) & 39 & 740 & 12464 & 290305 & ? \\
RHS: $\#$ L. orgraphs,  & 1\:(Jac) & 13 & 
241 
& ? & ? & ? \\
\phantom{RHS: } $\text{coeff} \neq 0$ & & & & \multicolumn{2}{c}{\upbracefill} \\
Reference & \S\ref{SecCoeffLeibniz}, \cite{KontsevichFormality} & \S\ref{ExLeibniz3}, \cite{sqs15} & \S\ref{ExLeibniz4}, \cite{cpp} & \multicolumn{2}{c}{\cite{BanksPanzerPym1812}}\\
\hline
\end{tabular}
\end{table}
At~$\hbar^4$, the expansion of $\Assoc(\star)(\cP)$ mod~$\bar{o}(\hbar^4)$ requires~$241$ 
nonzero coefficients of Leibniz graphs on $3$~sinks, $2=n-1$~internal vertices for bi\/-\/vectors~$\cP$ and one internal vertex for the tri-vector $\schouten{\cP,\cP}$, and therefore, $2(n-1)+3 = 2n+3-2=7$ oriented edges.

\begin{rem}
Again, 
this set of Leibniz graphs is well structured. Indeed, it is a disjoint union of homogeneous differential operators arranged according to their differential orders w.r.t.\ the sinks, e.g., $(1,1,1)$, $(2,1,1)$, $(1,2,1)$, $(1,1,2)$, etc., up to~$(3,3,1)$.
\end{rem}

\begin{example}\label{ExA331}
The Leibniz graph $L_{331}\mathrel{{:}{=}}
[\mathsf{01};\mathsf{01};\mathsf{012}]$ 
of differential orders $(3,3,1)$ has the weight $1/24$ according to~\cite{BanksPanzerPym1812}.
Multiplied by a universal (for all graphs at~$\hbar^4$) factor $2^4 = 16$ and the factor $1/(\# \operatorname{Aut}(L_{331})) = 1/2$ due to this graph's symmetry $(3 \rightleftarrows 4)$, it expands to $\tfrac{1}{3}\big([\mathsf{01};\mathsf{01};\mathsf{01};\mathsf{52}] + [\mathsf{01};\mathsf{01};\mathsf{12};\mathsf{50}] + [\mathsf{01};\mathsf{01};\mathsf{20};\mathsf{51}]\big)$ by the definition of Jacobi's identity.
This sum of three weighted Kontsevich orgraphs reproduces exactly $\mathsf{A}^{(4)}_{331}$, which is known from~\cite[Table~8 in App.~D]{cpp}.
\end{example}

\begin{example}\label{ExA322}
The Leibniz graph $L_{322}\mathrel{{:}{=}}
[\mathsf{01};\mathsf{02};\mathsf{012}]$
of differential orders $(3,2,2)$ has the weight $1/24$ according to~\cite{BanksPanzerPym1812}.
Multiplied now by a universal (for all graphs at~$\hbar^4$) factor $2^4 = 16$ and the factor $1/(\# \operatorname{Aut}(L_{322})) = 1$, it expands to $\tfrac{2}{3}\big([\mathsf{01};\mathsf{02};\mathsf{01};\mathsf{52}] + [\mathsf{01};\mathsf{02};\mathsf{12};\mathsf{50}] + [\mathsf{01};\mathsf{02};\mathsf{20};\mathsf{51}]\big)$.
This sum reproduces 
$\mathsf{A}^{(4)}_{322}$
(again, see
~\cite[Table~8 in App.~D]{cpp}).
\end{example}

\begin{example}\label{ExA132}
Consider at the differential order $(1,3,2)$ at $\hbar^4$ the three Leibniz graphs $L_{132}^{(1)}\mathrel{{:}{=}}[\mathsf{12};\mathsf{13};\mathsf{012}]$, $L_{132}^{(2)}\mathrel{{:}{=}}[\mathsf{12};\mathsf{12};\mathsf{014}]$, and $L_{132}^{(3)}\mathrel{{:}{=}}[\mathsf{12};\mathsf{01};\mathsf{412}]$.
They have no symmetries, i.e.\ their automorphism groups are one\/-\/element,
and their weights are $W(L_{132}^{(1)})=1/72$, $W(L_{132}^{(2)})=1/48$, and $W(L_{132}^{(3)})=1/48$, respectively.
Pre\/-\/multiplied by their weights and 
universal factor $2^4 = 16$, these Leibniz graphs expand to
\begin{align*}
\tfrac{2}{9}&\big([\mathsf{12};\mathsf{13};\mathsf{01};\mathsf{52}] + [\mathsf{12};\mathsf{13};\mathsf{12};\mathsf{50}] + [\mathsf{12};\mathsf{13};\mathsf{20};\mathsf{51}]\big) \\
&\quad{}
+\tfrac{1}{3}\big([\mathsf{12};\mathsf{12};\mathsf{01};\mathsf{54}] + [\mathsf{12};\mathsf{12};\mathsf{14};\mathsf{50}] + [\mathsf{12};\mathsf{12};\mathsf{40};\mathsf{51}]\big) \\
&\quad{}
+\tfrac{1}{3}\big([\mathsf{12};\mathsf{01};\mathsf{41};\mathsf{52}] + [\mathsf{12};\mathsf{01};\mathsf{12};\mathsf{54}] + [\mathsf{12};\mathsf{01};\mathsf{24};\mathsf{51}]\big).
\end{align*}
There is one cancellation, since $[\mathsf{12};\mathsf{01};\mathsf{12};\mathsf{54}] = -[\mathsf{12};\mathsf{12};\mathsf{01};\mathsf{54}]$.
The remaining seven terms reproduce exactly $\mathsf{A}^{(4)}_{132}$; that component is known from~\cite[Table~8 in App.~D]{cpp}.

Actually, there was another Leibniz graph at this homogeneity order,
$L_{132}^{(4)}\mathrel{{:}{=}} [\mathsf{12};\mathsf{15};\mathsf{012}]$, but its weight is zero
and hence it does not contribute.
(Indeed, we get an independent verification of this by having already balanced the entire homogeneous component at differential orders $(1,3,2)$ in the associator.)
\end{example}


\subsection*{Intermediate conclusion}

We have experimentally found the constants $c_k$ in Corollary~\ref{CorDiamondF} which balance the Kontsevich graph expansion of the $\hbar^k$-\/term $\mathsf{A}^{(k)}$ in the associator
against an expansion of the respective term at~$\hbar^k$ in the r.\/-\/h.s.\ of~\eqref{EqDiamondAssoc} using the weighted Leibniz graphs.
Namely, we conjecture $c_k = k/6$ in \S\ref{SecCoeffLeibniz}.
The origin of these constants, in particular how they arise from the sum over $i < j$ in the $L_\infty$ condition~\eqref{EqLinfty} (perhaps, in combination with different normalizations of the objects which we consider) still remains to be explained, similar 
to the reasoning in~\cite{ArnalManchonMasmoudi,WillwacherCalaque} where the signs are fixed.
Note that both in the associator, which is quadratic w.r.t.\ the weights of Kontsevich graphs in~$\star$, and in the operator~$\Diamond$, which is linear in the Kontsevich weights of Leibniz graphs, the weight values are provided simultaneously, by using identical techniques (for instance, from~\cite{BanksPanzerPym1812}).
Indeed, the weights are provided by the integral formula which is universal with respect to all the graphs under study~\cite{KontsevichFormality}.


\appendix
\section{Encodings of weighted Kontsevich\/-\/graph expansions for $(p,q,r)$-\/homogeneous components $(I_f,\ldots,S_h)_{pqr}$}
\label{AppIfShIntoK}
\noindent
\begin{verbatim}
# 2/3 (S_f)_{221}
3 3 1   0 1 0 1 4 2    2/3
3 3 1   0 1 1 2 4 0    2/3
3 3 1   0 1 2 0 4 1    2/3
# 2/3 (S_g)_{122}
3 3 1   1 2 0 1 4 2    2/3
3 3 1   1 2 1 2 4 0    2/3
3 3 1   1 2 2 0 4 1    2/3
# -2/3 (S_h)_{212}
3 3 1   2 0 0 1 4 2    -2/3
3 3 1   2 0 1 2 4 0    -2/3
3 3 1   2 0 2 0 4 1    -2/3
# 1/6 (I_f)_{111}
3 3 1   0 4 3 1 4 2    1/6
3 3 1   0 4 1 2 4 3    1/6
3 3 1   0 4 2 3 4 1    1/6
3 3 1   0 5 3 1 4 2    1/6
3 3 1   0 5 1 2 4 3    1/6
3 3 1   0 5 2 3 4 1    1/6
# -1/6 (I_h)_{111}
3 3 1   2 4 0 1 4 3    -1/6
3 3 1   2 4 1 3 4 0    -1/6
3 3 1   2 4 3 0 4 1    -1/6
3 3 1   2 5 0 1 4 3    -1/6
3 3 1   2 5 1 3 4 0    -1/6
3 3 1   2 5 3 0 4 1    -1/6
# 1/6 (I_f)_{112}
3 3 1   0 2 3 1 4 2    1/6
3 3 1   0 2 1 2 4 3    1/6
3 3 1   0 2 2 3 4 1    1/6
# 1/6 (I_g)_{112}
3 3 1   1 2 0 3 4 2    1/6
3 3 1   1 2 3 2 4 0    1/6
3 3 1   1 2 2 0 4 3    1/6
# -1/3 (S_h)_{112}
3 3 1   2 4 0 1 4 2    -1/3
3 3 1   2 4 1 2 4 0    -1/3
3 3 1   2 4 2 0 4 1    -1/3
3 3 1   2 5 0 1 4 2    -1/3
3 3 1   2 5 1 2 4 0    -1/3
3 3 1   2 5 2 0 4 1    -1/3
# 1/3 (I_f)_{121}
3 3 1   0 1 3 1 4 2    1/3
3 3 1   0 1 1 2 4 3    1/3
3 3 1   0 1 2 3 4 1    1/3
# -1/3 (I_h)_{121}
3 3 1   2 1 0 1 4 3    -1/3
3 3 1   2 1 1 3 4 0    -1/3
3 3 1   2 1 3 0 4 1    -1/3
# 1/3 (S_f)_{211}
3 3 1   0 4 0 1 4 2    1/3
3 3 1   0 4 1 2 4 0    1/3
3 3 1   0 4 2 0 4 1    1/3
3 3 1   0 5 0 1 4 2    1/3
3 3 1   0 5 1 2 4 0    1/3
3 3 1   0 5 2 0 4 1    1/3
# -1/6 (I_g)_{211}
3 3 1   1 0 0 3 4 2    -1/6
3 3 1   1 0 3 2 4 0    -1/6
3 3 1   1 0 2 0 4 3    -1/6
# -1/6 (I_h)_{211}
3 3 1   2 0 0 1 4 3    -1/6
3 3 1   2 0 1 3 4 0    -1/6
3 3 1   2 0 3 0 4 1    -1/6
\end{verbatim}

\subsubsection*{Acknowledgements}
The first author thanks the Organisers of international workshop 
`Symmetries \&\ integrability of equations of Mathematical Physics'
({22--24}~December \textup{2018}, IM NASU Kiev, Ukraine)
for helpful discussions and warm atmosphere during the meeting.
A part of this research was done while RB was visiting at~RUG and AVK was visiting at JGU~Mainz (supported by IM~JGU via project 5020 and JBI~RUG project~106552).
The research of AVK~is supported by the $\smash{\text{IH\'ES}}$ (partially, by the Nokia Fund). 

\end{document}